\documentclass[a4paper,12pt]{amsart}
\usepackage[english]{babel}
\usepackage[T1]{fontenc}
\usepackage[left=1.2in,right=1.2in,top=1.2in,bottom=1.2in]{geometry}
\usepackage{times}
\usepackage{microtype}

\usepackage{commath,amssymb,amscd,mathrsfs,mathtools,dsfont,bbm,multirow,array,caption,comment,pifont}
\usepackage{mathabx}
\usepackage[all]{xy}
\usepackage{enumitem}
\usepackage{longtable}

\usepackage{cite}			

\allowdisplaybreaks

\usepackage[dvipsnames,svgnames,table]{xcolor}
\definecolor{red}{RGB}{255,25,25}
\definecolor{blue}{RGB}{25,50,200}
\usepackage{tikz-cd}
\usepackage[pagebackref,linktocpage]{hyperref}

\hypersetup{
pdftitle={},				
pdfauthor={Fei Hu},	
colorlinks=true,			
linkcolor=red,			
citecolor=MidnightBlue,	
filecolor=magenta,		
urlcolor=MidnightBlue	
}

\usepackage{cleveref}	

\newtheorem{theorem}{Theorem}[section]
\crefname{theorem}{Theorem}{Theorems}
\newtheorem{lemma}[theorem]{Lemma}
\crefname{lemma}{Lemma}{Lemmas}
\newtheorem{proposition}[theorem]{Proposition}
\crefname{proposition}{Proposition}{Propositions}

\crefname{prop}{Proposition}{Propositions}
\newtheorem{corollary}[theorem]{Corollary}
\crefname{corollary}{Corollary}{Corollaries}

\crefname{cor}{Corollary}{Corollaries}
\newtheorem{conjecture}[theorem]{Conjecture}
\crefname{conjecture}{Conjecture}{Conjectures}

\crefname{conj}{Conjecture}{Conjectures}
\newtheorem*{conj*}{Conjecture}
\crefname{conj}{Conjecture}{Conjectures}

\theoremstyle{definition}
\newtheorem{definition}[theorem]{Definition}
\crefname{definition}{Definition}{Definitions}

\crefname{defn}{Definition}{Definitions}
\newtheorem{example}[theorem]{Example}
\crefname{example}{Example}{Examples}

\crefname{notation}{Notation}{Notation}
\newtheorem*{notation*}{Notation}
\crefname{notation}{Notation}{Notation}
\newtheorem*{convention*}{Convention}
\crefname{convention}{Convention}{Convention}

\crefname{problem}{Problem}{Problems}
\newtheorem{question}[theorem]{Question}
\crefname{question}{Question}{Questions}

\crefname{condition}{Condition}{Conditions}

\crefname{assumption}{Assumption}{Assumptions}

\theoremstyle{remark}

\crefname{rmk}{Remark}{Remarks}
\newtheorem*{rmk*}{Remark}
\crefname{rmk}{Remark}{Remarks}
\newtheorem{remark}[theorem]{Remark}
\crefname{remark}{Remark}{Remarks}

\crefname{fact}{Fact}{Facts}

\crefname{claim}{Claim}{Claims}
\newtheorem*{claim*}{Claim}
\crefname{claim}{Claim}{Claims}

\crefname{step}{Step}{Steps}

\crefname{case}{Case}{Cases}

\numberwithin{equation}{section}

\newcommand{\ol}[1]{\overline{#1}}

\newcommand{\arxiv}[1]{\href{https://arxiv.org/abs/#1}{{\tt arXiv:#1}}}
\newcommand{\doi}[1]{\href{https://doi.org/#1}{{\tt doi:#1}}}
\def\MR#1{\href{http://www.ams.org/mathscinet-getitem?mr=#1}{MR#1}}

\newcommand{\bC}{\mathbf{C}}

\newcommand{\bF}{\mathbf{F}}

\newcommand{\bQ}{\mathbf{Q}}
\newcommand{\bR}{\mathbf{R}}

\newcommand{\bZ}{\mathbf{Z}}

\newcommand{\bk}{\mathbbm{k}}

\newcommand{\sT}{\mathsf{T}}

\newcommand{\Sp}{\mathrm{Sp}}

\newcommand{\alb}{\operatorname{alb}}
\newcommand{\Alb}{\operatorname{Alb}}

\newcommand{\CH}{\mathsf{CH}}

\newcommand{\Char}{\operatorname{char}}

\newcommand{\cl}{\operatorname{cl}}

\newcommand{\et}{{\textrm{\'et}}}

\newcommand{\Frob}{\mathsf{Frob}}

\newcommand{\HP}{\mathsf{HP}}
\newcommand{\id}{\operatorname{id}}

\newcommand{\isom}{\simeq}

\newcommand{\NP}{\mathsf{NP}}

\newcommand{\Qbar}{\overline{\mathbf{Q}}}

\newcommand{\Zbar}{\overline{\mathbf{Z}}}

\begin{document}

\title[Parity and symmetry of polarized endomorphisms]{Parity and symmetry of polarized endomorphisms on cohomology}

\author{Fei Hu}
\address{School of Mathematics, Nanjing University, Nanjing 210093, Jiangsu, China}
\email{\href{mailto:fhu@nju.edu.cn}{\tt fhu@nju.edu.cn}}


\begin{abstract}
We show that the eigenvalues of any polarized endomorphism acting on the $\ell$-adic \'etale cohomology of a smooth projective variety satisfy certain parity and symmetry properties, as predicted by the standard conjectures.
These properties were previously known for Frobenius endomorphisms.
Besides the hard Lefschetz theorem, a key new ingredient is a recent Weil's Riemann hypothesis-type result due to J.~Xie.
We also prove a "Newton over Hodge" type property for abelian varieties and Grassmannians.
\end{abstract}

\subjclass[2020]{
14A10, 
14G17, 
14F20, 
14F40, 
15A63. 
}

\keywords{Frobenius endomorphism, polarized endomorphism, \'etale cohomology, de Rham cohomology, alternating form, symmetric form, parity, Newton polygon, Hodge polygon}


\maketitle

\section{Introduction}

We work over an algebraically closed field $\bk$ of characteristic $p\geq 0$.
Let $\ell$ be a prime number distinct from $p$.
Let $q\in \bZ_{>1}$ be an integer greater than $1$, which is not necessarily a power of $p$ and may have distinct prime factors.
An algebraic integer $\lambda\in \Zbar$ is called a \textit{$q$-Weil integer of weight $w$}, if for every field embedding $\iota\colon \bQ(\lambda) \to \bC$, we have
\[
|\iota(\lambda)|^2 = \iota(\lambda) \cdot \overline{\iota(\lambda)} = q^w.
\]
Motivated, presumably, by the K\"ahler analog of Weil's Riemann hypothesis established by Serre \cite{Serre60}, Tate proposed the following lesser-known conjecture in his 1963 Purdue talk (see also his 1964 Woods Hole talk \cite{Tate-WoodsHole}), which has been somewhat obscured by the prominence of the Tate conjecture.

\begin{conjecture}[{cf.~\cite[\S3, Conjecture (d)]{Tate65}}]
\label{conj:Tate-d}
Let $X$ be a smooth projective variety of dimension $d$ over $\bk$, and let $f$ be a ($q$-)polarized endomorphism of $X$, i.e., there is an ample divisor $H_X$ on $X$ such that $f^*H_X \sim q H_X$ for some integer $q>1$.
Then, for each $0 \leq i \leq 2d$, the following assertions hold:
\begin{enumerate}[label=\emph{(\arabic*)}, ref=(\arabic*)]
\item \label{conj:Tate-WRH} All eigenvalues of $f^*|_{H^i_{\emph{\et}}(X, \bQ_\ell)}$ are $q$-Weil integers of weight $i$.
\item \label{conj:Tate-SS} The linear transformation $f^*|_{H^i_{\emph{\et}}(X, \bQ_\ell)}$ is semisimple.
\end{enumerate}
\end{conjecture}

\begin{remark}
\label{rmk:integrality}
\begin{enumerate}[itemindent=2.5em, leftmargin=0pt, rightmargin=0cm]
\item As Tate observed, in characteristic zero, \cref{conj:Tate-d} follows from Hodge theory by Serre's pioneering work \cite{Serre60}.
When $\bk = \ol\bF_q$ and $f$ is the geometric Frobenius endomorphism $\Frob_q$ of $X$, part \ref{conj:Tate-WRH} of \cref{conj:Tate-d} corresponds to Weil's Riemann hypothesis, which was proven by Deligne \cite{Deligne74}.
However, its part \ref{conj:Tate-SS}, known as the semisimplicity conjecture, remains open in general.
It is well known that \cref{conj:Tate-d} follows from the standard conjectures (see, e.g., \cite[\S 4]{Kleiman68}).

\item Building on Deligne's groundbreaking work \cite{Deligne74}, Katz and Messing \cite{KM74} proved that the standard conjecture $C$ holds for all smooth projective varieties $Y$ over finite fields.
As a consequence, the characteristic polynomial of any endomorphism $g$ of $Y$ acting on any Weil cohomology group $H^*(Y)$ has integer coefficients and is independent of the choice of Weil cohomology.
Then, by applying a standard spreading-out and specialization argument, it follows that all eigenvalues appearing in \cref{conj:Tate-d}\ref{conj:Tate-WRH} are algebraic integers.
\end{enumerate}
\end{remark}

In a recent preprint \cite{Xie-GWRH}, Xie has made substantial progress on \cref{conj:Tate-d}, fully proving its longstanding open part \ref{conj:Tate-WRH}, which had remained unresolved for 60 years.
His proof relies on Deligne's resolution of Weil's Riemann hypothesis.
In particular, his result below simultaneously extends \cite[Theorem~1.6]{Deligne74} on the eigenvalues of the Frobenius endomorphisms to arbitrary polarized endomorphisms and \cite[Th\'eor\`eme~1]{Serre60} to arbitrary characteristic, as anticipated by the standard conjectures.

\begin{theorem}[{cf. \cite[Corollary~1.5]{Xie-GWRH}}]
\label{thm:Xie}
\cref{conj:Tate-d}\ref{conj:Tate-WRH} is true.
\end{theorem}

From a motivic perspective, polarized endomorphisms exhibit behavior analogous to that of Frobenius endomorphisms.
For instance, the standard conjectures imply both Weil’s Riemann hypothesis and the semisimplicity of not only Frobenius endomorphisms but also all polarized endomorphisms (see \cite[\S 4]{Kleiman68}).
In this short note, we further supplement that certain parity and symmetry properties of Frobenius endomorphisms also extend to polarized endomorphisms.
For details on these properties, see \cite{Suh12, SZ16}.

Before stating the result, we introduce some notation.
Let $X$ be a smooth projective variety of dimension $d$ over $\bk$.
Assume that $f$ is an arbitrary endomorphism of $X$.
For any $0\leq i\leq 2d$, denote the algebraic multiplicity of an eigenvalue $\lambda$ of $f^*|_{H^i_{\et}(X, \bQ_\ell)}$ by $\mu_\lambda\in \bZ_{\geq 0}$.
For each $e\geq 1$, let $\mu_{\lambda, e}\in \bZ_{\geq 0}$ denote the number of $e\times e$ Jordan blocks associated with the eigenvalue $\lambda$ in the Jordan canonical form of $f^*|_{H^i_{\et}(X, \bQ_\ell)}$.
Denote the characteristic polynomial of $f^*|_{H^i_{\et}(X, \bQ_\ell)}$ by
\[
P_i(t) \coloneqq \det (t \id - f^*|_{H^i_{\et}(X, \bQ_\ell)}) = \prod_{\lambda} \, (t-\lambda)^{\mu_\lambda} \in \bZ[t].
\]
Note that $\mu_\lambda = \sum_{e\geq 1} e \mu_{\lambda, e}$, and the geometric multiplicity of $\lambda$ equals $\sum_{e\geq 1} \mu_{\lambda, e}$.

\begin{theorem}
\label{thm:parity-symmetry}
Let $X$ be a smooth projective variety of dimension $d$ over $\bk$, and $f$ a $q$-polarized endomorphism of $X$, i.e., there is an ample divisor $H_X$ on $X$ such that $f^*H_X \sim q H_X$ for some $q\in \bZ_{>1}$.
Let $\ell$ be a prime distinct from $p$.
Then the following assertions hold:
\begin{enumerate}[label=\emph{(\arabic*)}, ref=(\arabic*)]
\item \label{thm:multiplicity} For any $0\leq i\leq 2d$, any eigenvalue $\lambda$ of $f^*|_{H_{\emph{\et}}^i(X, \bQ_\ell)}$, and any $e\geq 1$, we have $\mu_{\lambda, e} = \mu_{q^i/\lambda, e}$;
in particular, $\mu_{\lambda} = \mu_{q^i/\lambda}$.

\item \label{thm:functional-eq} For any $0\leq i\leq 2d$, the characteristic polynomial $P_i(t)$ satisfies the functional equation
\[
t^{b_i} P_i(q^i/t) = (-1)^{\epsilon_i} q^{\frac{ib_{i}}{2}} P_i(t),
\]
where $b_i \coloneqq b_i(X)\coloneqq \dim_{\bQ_\ell} \! H^i_{\emph{\et}}(X, \bQ_\ell)$ is the $i$th Betti number of $X$ and $\epsilon_i\in\{0,1\}$.

\item \label{thm:evenness} For any odd $0<i<2d$, the real roots $\pm q^{i/2}$ of $P_i(t)$ have even multiplicity, i.e., $\mu_{\pm q^{i/2}}$ are even (possibly zero).
In particular, for odd $i$, the $\epsilon_i$ in the functional equation is zero.
\end{enumerate}
\end{theorem}

\begin{remark}
\label{rmk:parity-symmetry}
\begin{enumerate}[itemindent=2.5em, leftmargin=0pt, rightmargin=0cm]
\item The above assertion~\ref{thm:evenness} provides a partial answer to a question posed by the author in 2019, which, as Deligne kindly pointed out, also follows from the standard conjectures (see \cite[Remark~1.3]{Hu24}).
When $\bk = \ol\bF_q$ and $f$ is the geometric Frobenius endomorphism $\Frob_q$ of $X$, \cref{thm:parity-symmetry} is a well-known consequence of Deligne's Weil II \cite{Deligne80}.
Moreover, these parity and symmetry properties hold even for the Frobenius endomorphism of \textit{proper} varieties.
For further details, see \cite{Suh12,SZ16}.

\item It is worth mentioning that by Poincaré duality one can easily derive another functional equation for characteristic polynomials $P_{i}$ and $P_{2d-i}$ as follows:
\[
t^{b_i} P_i(q^d/t) = (-1)^{\epsilon_i} q^{\frac{ib_{i}}{2}} P_{2d-i}(t),
\]
where the sign is exactly the same as in the assertion~\ref{thm:functional-eq}.
\end{enumerate}
\end{remark}

As a direct consequence of \cref{thm:parity-symmetry}, we describe the nonarchimedean valuations of the eigenvalues of polarized endomorphisms acting on cohomology, or equivalently, the Newton polygons of the corresponding characteristic polynomials.

Recall that for a polynomial
\[
P(t) = a_0 t^n + a_{1} t^{n-1} + \dots + a_{n} \in K[t]
\]
with coefficients $a_i$ in a nonarchimedean field $(K, \nu)$, the \textit{Newton polygon $\NP_{\nu}(P)$} of $P(t)$ is defined as the lower boundary of the convex hull of the set of points $(i, \nu(a_i))$ with $a_i\neq 0$ in $\bR^2$.
This polygon fully encodes the $\nu$-valuations of the roots of $P(t)$.
Indeed, if the Newton polygon $\NP_{\nu}(P)$ has a segment of slope $\lambda$ and (horizontal) length $l$, then $P(t)$ has precisely $l$ roots (counted with multiplicity) with $\nu$-valuation $\lambda$.

\begin{corollary}
\label{cor:non-arch-NP}
Let $X$ be a smooth projective variety of dimension $d$ over $\bk$, and $f$ a $q$-polarized endomorphism of $X$ with $\Char \bk = p \geq 0$ and $q\in \bZ_{>1}$.
Let $\ell$ be a prime distinct from $p$.
Denote by $\nu_p$ (resp. $\nu_\ell$) a $p$-adic (resp. $\ell$-adic) valuation on $\Qbar$.
For any $0\leq i\leq 2d$, denote the characteristic polynomial of $f^*|_{H^i_{\emph{\et}}(X, \bQ_\ell)}$ by $P_i(t)\in \bZ[t]$.
Then the following assertions hold:
\begin{enumerate}[label=\emph{(\arabic*)}, ref=(\arabic*)]
\item \label{cor:ell-nmid-q} If $\ell \nmid q$, then the Newton polygon $\NP_{\nu_\ell}(P_i)$ has a single segment of slope $0$ and length $b_i$, i.e., all eigenvalues of $f^*|_{H_{\emph{\et}}^i(X, \bQ_\ell)}$ are $\ell$-adic units.

\item \label{cor:p-nmid-q} If $p>0$ and $p \nmid q$, then $\NP_{\nu_p}(P_i)$ has a single segment of slope $0$ and length $b_i$.

\item \label{cor:ell-mid-q} If $\ell \! \mid \! q$ and $\nu_\ell$ is normalized so that $\nu_\ell(q)=1$, then $\NP_{\nu_\ell}(P_i)$ is symmetric, i.e., its slopes are rational numbers in $[0,i]$, and each slope $\alpha$ appears with the same length as $i-\alpha$.

\item \label{cor:p-mid-q} If $p>0$, $p \! \mid \! q$, and $\nu_p$ is normalized so that $\nu_p(q)=1$, then $\NP_{\nu_p}(P_i)$ is symmetric.
\end{enumerate}
\end{corollary}

A sophisticated reader may wonder what lies beyond the symmetry in \cref{cor:non-arch-NP}\ref{cor:ell-mid-q} and \ref{cor:p-mid-q}.
Indeed, Katz conjectured, first proven by Mazur \cite{Mazur72,Mazur73} and later extended by Ogus \cite[\S8]{BO78}, that the Newton polygon of the characteristic polynomial of the geometric Frobenius $\Frob_q$ on $H^i_{\et}(X, \bQ_\ell)$ lies on or above the Hodge polygon $\HP_i(X)$ of $X$ of weight $i$.

Recall that for a smooth projective variety $X$ of dimension $d$ over $\bk$, and an integer $0\leq i\leq 2d$, the \textit{Hodge polygon} $\HP_i(X)$ of $X$ of weight $i$ is a piecewise linear (convex) function passing through the origin $(0,0)$ and the points
\[
\Bigg(\sum_{j=0}^k h^{j,i-j}, \sum_{j=0}^k j h^{j,i-j}\Bigg), \quad k=0,1,\dots,i,
\]
where $h^{j,i-j}$ are the Hodge numbers, defined as
\[
h^{j,i-j} \coloneqq h^{j,i-j}(X) \coloneqq \dim_{\bk} H^{i-j}(X, \Omega_{X/\bk}^j).
\]
Equivalently, $\HP_i(X)$ has slope zero over the interval $(0, h^{0,i})$ and slope $k=1,2,\dots,i$ over the interval 
\[
(h^{0,i}+\dots+h^{k-1,i-k+1}, h^{0,i}+\dots+h^{k,i-k}).
\]
As an analog of Katz's conjecture or the Mazur--Ogus theorem (see \cite[Theorem~8.39]{BO78}), it is natural to ask the following question.

\begin{question}
\label{qn:NP vs HP}
Let $X$ be a smooth projective variety of dimension $d$ over $\bk$, and $f$ a $q$-polarized endomorphism of $X$ with $\Char \bk = p \geq 0$ and $q\in \bZ_{>1}$.
Let $\ell$ be a prime distinct from $p$.
Assume that $\ell \! \mid \! q$ (resp. $p>0$, $p \! \mid \! q$), and let $\nu_\ell$ (resp. $\nu_p$) denote the normalized $\ell$-adic (resp. $p$-adic) valuation on $\Qbar$ so that $\nu_\ell(q)=1$ (resp. $\nu_p(q)=1$).
Then, for any $0\leq i\leq 2d$, is it true that
\[
\NP_{\nu_\ell}(P_i) \geq \HP_i(X) \ \ \text{(resp. } \NP_{\nu_p}(P_i) \geq \HP_i(X) \text{)}?
\]
\end{question}

\begin{remark}
\begin{enumerate}[itemindent=2.5em, leftmargin=0pt, rightmargin=0cm]
\item If $\bk = \ol\bF_q$ and $f$ is the geometric Frobenius endomorphism $\Frob_q$ of $X$ (in particular, $q$ is a power of some prime $p$), then the inequality $\NP_{\nu_p}(P_i) \geq \HP_i(X)$ is precisely the Mazur--Ogus theorem, whose proof relies on crystalline cohomology.
Hence, for a general polarized endomorphism $f$, although its characteristic polynomial $P_i(t)$ is independent of the chosen Weil cohomology theory, it may still be useful to study the induced action on crystalline cohomology.

\item The above \cref{qn:NP vs HP} remains compelling even in characteristic zero, as it imposes conjectural numerical constraints on polarized endomorphisms, whose classification in characteristic zero is still largely open (see, e.g., \cite{MZ18a}).

\item By the Lefschetz trace formula, an affirmative answer to \cref{qn:NP vs HP} would also yield an analogous divisibility statement for the number of periodic points of $f$, counted with multiplicity (cf. \cite[Divisibility Corollary]{Mazur72}).
\end{enumerate}
\end{remark}

\begin{proposition}
\label{prop:AV&Gr}
\cref{qn:NP vs HP} admits an affirmative answer in the cases of abelian varieties and Grassmannians.
\end{proposition}

\begin{remark}
For any $q$-polarized endomorphism $f$ of an irregular smooth projective variety $X$ over $\bk$, using the fact that the Albanese morphism $\alb_X \colon X\to \Alb(X)$ is $f$-equivariant, \cref{qn:NP vs HP} also admits an affirmative answer for $i=1, 2d-1$ (see also \cite[Corollary~1.4]{Hu24}).
\end{remark}

\section{Proof of main results}

\subsection{Proof of Theorem \ref{thm:parity-symmetry} and Corollary \ref{cor:non-arch-NP}}

Keep in mind the assumption and the notation of \cref{thm:parity-symmetry}.
For any $0\leq i\leq d$, consider the following bilinear form
\begin{equation}
\label{eq:bilinear-form}
\begin{array}{ccc}
B\colon H^i_{\et}(X, \bQ_\ell) \times H^i_{\et}(X, \bQ_\ell) & \to & H^{2d}_{\et}(X, \bQ_\ell) \isom \bQ_\ell \\
\qquad (\alpha, y) & \mapsto & \alpha \cup y \cup \cl_X(H_X)^{d-i},
\end{array}
\end{equation}
where $\cl_X\colon \CH^i(X) \to H^{2i}_{\et}(X, \bQ_\ell)$ is the cycle class map.
By Poincaré duality and the hard Lefschetz theorem \cite[Théorème~4.1.1]{Deligne80}, $B$ is nondegenerate;
moreover, $B$ is alternating (resp. symmetric), if $i$ is odd (resp. even).

\begin{lemma}
\label{lemma:pairing}
For each $0\leq i\leq d$, let $\varphi_i$ denote the following normalized linear transformation
\[
q^{-i/2} f^*|_{H^i_{\emph{\et}}(X, \bQ_\ell)}
\]
on $H^i_{\emph{\et}}(X, \bQ_\ell)$.
Then the bilinear form $B$ defined in \eqref{eq:bilinear-form} is preserved by $\varphi_i$, i.e.,
for any $\alpha, \beta\in H^i_{\emph{\et}}(X, \bQ_\ell)$, we have
\[
B(\varphi_i(\alpha), \varphi_i(\beta)) = B(\alpha, \beta).
\]
In particular, $\varphi_i$ is symplectic (resp. orthogonal), if $i$ is odd (resp. even).
\end{lemma}
\begin{proof}
Since $f^*H_X \sim qH_X$, it is well known that such an $f$ is a finite morphism of degree $\deg(f)=q^d$.
By the projection formula, we have for any $\alpha, \beta\in H^i_{\et}(X, \bQ_\ell)$,
\begin{align*}
B(f^*\alpha, f^*\beta) & = f^*\alpha \cup f^*\beta \cup \cl_X(H_X)^{d-i} \\
& = f^*\alpha \cup f^*\beta \cup \frac{1}{q^{d-i}} \cl_X(f^*H_X)^{d-i} \\
& = \frac{\deg(f)}{q^{d-i}} \, \alpha \cup \beta \cup \cl_X(H_X)^{d-i} \\
& = q^i B(\alpha, \beta).
\end{align*}
Thus, it follows that $B(\varphi_i(\alpha), \varphi_i(\beta)) = q^{-i} B(f^*\alpha, f^*\beta) = B(\alpha, \beta)$.
\end{proof}

\begin{proof}[Proof of \cref{thm:parity-symmetry}]
According to a fundamental result of Katz and Messing \cite{KM74}, we now know that the characteristic polynomial $P_i(t)$ of $f^*|_{H^i_{\et}(X, \bQ_\ell)}$ has integer coefficients, whose roots are algebraic integers.
Note, by Poincaré duality, that the transpose of $f^*|_{H^{2d-i}_{\et}(X, \bQ_\ell)}$ is equal to $q^d (f^*|_{H^i_{\et}(X, \bQ_\ell)})^{-1}$.
Thus, it follows that
\[
P_{2d-i}(t) = \det (t \id - f^*|_{H^{2d-i}_{\et}(X, \bQ_\ell)}) = \frac{(-t)^{b_{i}}}{\det(f^*|_{H^i_{\et}(X, \bQ_\ell)})} P_i(q^d/t).
\]
In particular, the eigenvalues of $f^*|_{H^{2d-i}_{\et}(X, \bQ_\ell)}$ are precisely given by $q^d/\lambda$, where the  $\lambda$ are the eigenvalues of $f^*|_{H^{i}_{\et}(X, \bQ_\ell)}$.
Replacing $t$ by $q^{2d-i}/t$, one obtains that
\[
t^{b_{2d-i}} P_{2d-i}(q^{2d-i}/t) = \frac{(-1)^{b_{i}}q^{(2d-i)b_{i}}}{\det(f^*|_{H^i_{\et}(X, \bQ_\ell)})} P_i(t/q^{d-i}).
\]
Therefore, it suffices to prove \cref{thm:parity-symmetry} for $0\leq i\leq d$.

Fix any $0\leq i\leq d$ and consider the nondegenerate bilinear form $B$ defined in \eqref{eq:bilinear-form}.
Thanks to \cref{lemma:pairing}, the normalized linear transformation $\varphi_i$ of $H^i_{\et}(X, \bQ_\ell)$ preserves $B$.
It follows that $\varphi_i \in \Sp(H^i_{\et}(X, \bQ_\ell), B)$ (resp. $\mathrm{O}(H^i_{\et}(X, \bQ_\ell), B)$), the symplectic (resp. orthogonal) group of $(H^i_{\et}(X, \bQ_\ell), B)$, if $i$ is odd (resp. even).
In particular, the determinant of $\varphi_i$ is $\pm 1$, and the inverse $\varphi_i^{-1}$ of $\varphi_i$ is similar to the transpose $\varphi_i^\sT$ of $\varphi_i$.
Equivalently, the transpose of $f^*|_{H^i_{\et}(X, \bQ_\ell)}$ is similar to $q^i \, (f^*|_{H^i_{\et}(X, \bQ_\ell)})^{-1}$.
The assertion~\ref{thm:multiplicity} thus follows.

Let $P_{\varphi_i}(t)\in \bQ(\sqrt{q})[t]$ denote the characteristic polynomial of the linear map $\varphi_i$.
Then
\begin{align*}
P_{\varphi_i}(t) & = P_{\varphi_i^{-1}}(t) = \det(t \id - \varphi_i^{-1}) = \det(\varphi_i)^{-1} \det(t \, \varphi_i - \id) \\
& = \det(\varphi_i) (-t)^{b_i} \det(t^{-1} \id - \varphi_i) = \det(\varphi_i) (-t)^{b_i} P_{\varphi_i}(t^{-1}).
\end{align*}
On the other hand, by the definition of $\varphi_i$, one has 
$P_{\varphi_i}(t) = q^{-\frac{ib_{i}}{2}} P_i(q^{i/2} \, t)$.
The assertion~\ref{thm:functional-eq} thus follows.

Suppose now that $0<i\leq d$ is odd and hence $\varphi_i \in \Sp(H^i_{\et}(X, \bQ_\ell), B)$.
Then the determinant of $\varphi_i$ is one and the $i$-th Betti number $b_i$ is even.
In particular, the characteristic polynomial $P_{\varphi_i}(t)$ of $\varphi_i$ is self-reciprocal, i.e., $P_{\varphi_i}(t) = t^{b_i}P_{\varphi_i}(t^{-1})$.
Recall that the characteristic polynomial of $f^*|_{H^i_{\et}(X, \bQ_\ell)}$ has integer coefficients, whose roots are algebraic integers denoted by $\lambda_{i,j}$ with $1\leq j\leq b_i$ (see \cref{rmk:integrality}).
Xie's recent brilliant result shows that
\[
\lambda_{i,j} = q^{i/2} \, e^{\sqrt{-1} \, \theta_{i,j}}
\]
with $\theta_{i,j}\in [0, 2\pi)$ for all $j$ (see \cite[Corollary~1.5]{Xie-GWRH} or \cref{thm:Xie} which extends \cite[Th\'eor\`eme~1]{Serre60} to arbitrary characteristic).
Hence the eigenvalues of the normalized $\varphi_i$ are exactly these $e^{\sqrt{-1} \, \theta_{i,j}}$.
Note that the sum of the multiplicities of the two real eigenvalues $\pm 1$ is even, because $P_{\varphi_i}(t)\in \bQ(\sqrt{q})[t]$ and $\deg P_{\varphi_i}(t) = b_i$ is even.
The multiplicity of the eigenvalue $-1$ is also even, since the determinant of $\varphi_i$ is one.
It follows that the two real roots $\pm 1$ of $P_{\varphi_i}(t)$ have even multiplicity (possibly zero), so are the real eigenvalues $\pm q^{i/2}$ of $f^*|_{H^i_{\et}(X, \bQ_\ell)}$.
This completes the proof of \cref{thm:parity-symmetry}.
\end{proof}

\begin{remark}
\label{rmk:zeta}
It follows from the above proof that $\epsilon_i \equiv b_i + \mu_{-q^{i/2}} \pmod 2$ in the functional equation for any $0\leq i\leq 2d$; moreover, if $i$ is odd, then $\epsilon_i = 0$.
By Poincaré duality and the Lefschetz trace formula, one also obtains a similar functional equation for the \textit{Lefschetz dynamical zeta function} $Z_f(t)$ associated with $X$ and $f$ defined by $\exp(\sum_{n\geq 1} N_n t^n /n)$, where $N_n$ is the number of fixed points of $f^n$ (counted with multiplicity), as follows:
\[
Z_f(q^{-d}t^{-1}) = (-1)^{\chi + \mu_{-q^{d/2}}}  \, q^{d\chi/2} \, t^{\chi} Z(t),
\]
where $\chi \coloneqq \chi(X) \coloneqq \sum_{i=0}^{2d} (-1)^i b_i(X)$ is the Euler characteristic of $X$ and as before $\mu_{-q^{d/2}}$ denotes the multiplicity of $-q^{d/2}$ as an eigenvalue of $f^*|_{H^d_{\et}(X, \bQ_\ell)}$.
\end{remark}

\begin{proof}[Proof of \cref{cor:non-arch-NP}]
Recall again by Katz--Messing \cite{KM74}, the eigenvalues of $f^*|_{H^i_{\et}(X, \bQ_\ell)}$ are algebraic integers.
When $\ell \nmid q$, the $\ell$-adic valuations of the eigenvalues $\lambda$ of $f^*|_{H^i_{\et}(X, \bQ_\ell)}$ are nonnegative.
On the other hand, Poincaré duality implies that $q^d/\lambda$ is an eigenvalue of $f^*|_{H^{2d-i}_{\et}(X, \bQ_\ell)}$, and hence also has nonnegative $\ell$-adic valuations.
It follows that $\lambda$ is an $\ell$-adic unit in a suitable field extension over $\bQ_\ell$.

When $\ell \! \mid \! q$, thanks to Xie's \cref{thm:Xie}, if $\lambda$ is an eigenvalue of $f^*|_{H^i_{\et}(X, \bQ_\ell)}$, then so is its complex conjugate $\ol{\lambda}$, and moreover, $\lambda \cdot \ol{\lambda} = q^i$.
This implies that $\nu_\ell(\lambda) + \nu_\ell(\ol{\lambda}) = i$, and hence the Newton polygon $\NP_{\nu_\ell}(P_i)$ of the characteristic polynomial $P_i(t)$ is symmetric.

The remaining assertions for the Newton polygon $\NP_{\nu_p}(P_i)$ follow by similar reasoning.
\end{proof}

\subsection{Proof of Proposition \ref{prop:AV&Gr}}

Recall that for a monic polynomial $P(t) = t^{n} + a_1 t^{n-1} + \dots + a_n$ over a nonarchimedean valued field $(K, \nu)$, its Newton polygon $\NP_\nu(P)$ is completely characterized by a finite sequence of numbers
\[
s_1(P) \leq s_2(P) \leq \dots \leq s_n(P),
\]
where the $s_i(P)$ are called the \textit{slopes} of the polygon $\NP_\nu(P)$, corresponding to the $\nu$-valuations of the roots of $P(t)$.
In particular, the Newton polygon $\NP_\nu(P)$ is represented by the piecewise linear function $\varphi_P \colon [0,n]\to \bR$ such that $\varphi_P(0) = 0$ and for $k=1,2,\ldots,n$,
\[
\varphi_P(k) = s_1(P) + s_2(P) + \dots + s_k(P).
\]
Hence for any two monic polynomials $P(t), Q(t) \in K[t]$ of the same degree $n$, the associated Newton polygons satisfy $\NP_\nu(P) \geq \NP_\nu(Q)$ if and only if for all $k=1,2,\ldots,n$,
\[
\sum_{i=1}^{k} s_i(P) \geq \sum_{i=1}^{k} s_i(Q).
\]
This is reminiscent of the theory of majorization.

For any $x=(x_1, x_2, \ldots, x_n) \in \bR^n$, let
\[
x_{(1)} \leq x_{(2)} \leq \dots \leq x_{(n)} \quad \text{(resp. } x_{[1]} \geq x_{[2]} \geq \dots \geq x_{[n]} \text{)}
\]
denote the components of $x$ in nondecreasing (resp. nonincreasing) order.

\begin{definition}[{cf.~\cite[Definition~A.1]{MOA11}}]
\label{def:majorization}
For any $x, y \in \bR^n$, we say that $x$ is \textit{majorized} by $y$ (or $y$ \textit{majorizes} $x$) and write $x \prec y$, if for all $k=1,2,\dots,n-1$,
\begin{align*}
\label{eq:majorization}
\sum_{i=1}^k x_{[i]} \leq \sum_{i=1}^k y_{[i]}, \ \text{ and } \ 
\sum_{i=1}^n x_{[i]} = \sum_{i=1}^n y_{[i]}.
\end{align*}
Equivalently, $x \prec y$ if for all $k=1,2,\dots,n-1$,
\begin{align*}
\sum_{i=1}^k x_{(i)} \geq \sum_{i=1}^k y_{(i)}, \ \text{ and } \ 
\sum_{i=1}^n x_{(i)} = \sum_{i=1}^n y_{(i)}.
\end{align*}
\end{definition}

To show \cref{prop:AV&Gr} in the case of abelian varieties, we need the following classic result essentially due to Schur.

\begin{theorem}[{cf. \cite[Proposition~C.2]{MOA11}}]
\label{thm:Schur23}
If $\varphi \colon \bR^n \to \bR$ is a symmetric and convex function, then $\varphi$ is Schur-convex, i.e., $x \prec y$ implies $\varphi(x) \leq \varphi(y)$.
\end{theorem}

For any $x\in \bR^n$ and $k=1,2,\dots,n$, denote
\begin{equation}
\label{eq:compound}
x^{(k)} = (x_{i_1}+x_{i_2}+\dots+x_{i_k})_{1\leq i_1<i_2<\dots<i_k\leq n} \in \bR^{\binom{n}{k}},
\end{equation}
where the indices $(i_1,i_2,\dots,i_k)\in \{1,2,\dots,n\}^k$ are sorted in the following natural way: $(i_1,i_2,\dots,i_k) < (i'_1,i'_2,\dots,i'_k)$ if the left-most nonzero component of $(i_1-i'_1,i_2-i'_2,\dots,i_k-i'_k)$ is negative.
In particular, $x^{(1)}=x$, $x^{(n)}=x_1+x_2+\dots+x_n$, and
\[
x^{(2)}=(x_1+x_2, x_1+x_3, \dots, x_1+x_n, x_2+x_3, \dots, x_{n-1}+x_n)\in \bR^{\binom{n}{2}}.
\]
This gives rise to an $\bR$-linear map
\begin{align*}
c_k\colon \bR^n \to \bR^{\binom{n}{k}}, \quad x \mapsto x^{(k)}.
\end{align*}

\begin{lemma}
\label{lemma:wedge-product-preserves-majorization}
Let $x, y\in \bR^n$ such that $x \prec y$.
Then for any $k=2,3,\dots,n$, we have $x^{(k)} \prec y^{(k)}$.
\end{lemma}
\begin{proof}
Let $k$ be fixed.
Let $l$ be an arbitrary positive integer $\leq \binom{n}{k}$.
Define a function $s_l\colon \bR^{\binom{n}{k}} \to \bR$ as follows: for $z\in \bR^{\binom{n}{k}}$, $s_l(z)$ is the sum of the $l$ largest components of $z$, i.e.,
\[
s_l(z) = \sum_{j=1}^l z_{[j]}.
\]
Since $s_l$ is the maximum of all possible sums of $l$ different components of $z$, it is convex.
As $c_k$ is $\bR$-linear, the composite function $s_l\circ c_k \colon \bR^n \to \bR$ is also convex.
On the other hand, by definition, $s_l\circ c_k$ is symmetric.
Therefore, thanks to \cref{thm:Schur23}, $s_l\circ c_k$ is Schur-convex; in particular, we prove that $s_l(x^{(k)}) \leq s_l(y^{(k)})$.

When $l=\binom{n}{k}$, it is easy to see that
\begin{align*}
s_{l}(x^{(k)}) &= \sum_{1\leq i_1<i_2<\dots<i_k\leq n} (x_{i_1}+x_{i_2}+\dots+x_{i_k}) \\
&= \binom{n-1}{k-1} \sum_{i=1}^n x_i = \binom{n-1}{k-1} \sum_{i=1}^n y_i = s_{l}(y^{(k)}).
\end{align*}
Hence by \cref{def:majorization}, we have shown that $x^{(k)} \prec y^{(k)}$.
\end{proof}

\begin{proof}[Proof of \cref{prop:AV&Gr}]
Let $X$ be a smooth projective variety of dimension $d$ over $\bk$, and let $f$ be a $q$-polarized endomorphism of $X$.
Let $p,q,\ell,\nu_p,\nu_\ell$ be as in \cref{qn:NP vs HP}.

We first consider the case where $X$ is an abelian variety.
By \cref{cor:non-arch-NP}\ref{cor:ell-mid-q} and \ref{cor:p-mid-q}, the Newton polygons $\NP_{\nu_p}(P_1)$ and $\NP_{\nu_\ell}(P_1)$ are symmetric with slopes in $[0,1]\cap \bQ$.
Any such symmetric Newton polygon lies between the supersingular polygon (a single segment of slope $1/2$ and length $2d$) and the ordinary polygon (consisting of two segments connecting the points $(0,0)$, $(d,0)$, and $(2d,d)$).
The latter polygon coincides with the Hodge polygon $\HP_1(X)$ of weight $1$ by definition.
We thus prove that $\NP_{\nu_\ell}(P_1) \geq \HP_1(X)$ and $\NP_{\nu_p}(P_1) \geq \HP_1(X)$.
For $i\geq 2$, since $H^i_{\et}(X, \bQ_\ell)$ is canonically isomorphic with the wedge product $\wedge^i H^1_{\et}(X, \bQ_\ell)$, the characteristic polynomial $P_i(t)$ of $f^*|_{H^i_{\et}(X, \bQ_\ell)}$ is completely determined by $P_1(t)$.
Precisely, denote the slopes of $\NP_{\nu_\ell}(P_1)$ by $s(P_1)=(s_1(P_1),s_2(P_1),\dots,s_{2d}(P_1))$ in nondecreasing order.
Then the slopes of $\NP_{\nu_\ell}(P_i)$ are exactly $s(P_1)^{(i)}$ (see \cref{eq:compound}; note that $s(P_1)^{(i)}$ may not be in nondecreasing order).
On the other hand, by thinking of $\HP_1(X)$ as the Newton polygon of the polynomial $Q_1(t) = (t-1)^d(t-q)^d$ whose slopes are $s(Q_1)=(0,\dots,0,1,\dots,1)$, we see that the slopes of $\HP_i(X)$ are similarly given by $s(Q_1)^{(i)}$.
Since the Newton polygon $\NP_{\nu_\ell}(P_1)$ and the Hodge polygon $\HP_1(X)$ have the same end point $(2d, d)$, one has $s(P_1) \prec s(Q_1)$ (see \cref{def:majorization}).
It thus follows from \cref{lemma:wedge-product-preserves-majorization} that $s(P_1)^{(i)} \prec s(Q_1)^{(i)}$ for all $2\leq i\leq 2d$, i.e., $\NP_{\nu_\ell}(P_i) \geq \HP_i(X)$.
The assertion $\NP_{\nu_p}(P_i) \geq \HP_i(X)$ follows by similar reasoning.

Next, suppose that $X$ is the Grassmannian $G\coloneqq G(k, n)$ of $k$-dimensional subspaces in an $n$-dimensional vector space over $\bk$.
It is well known that $h^{i,j}(G) = 0$ for all $i\neq j$ and $h^{j,j}(G) = b_{2j}(G)$.
Fix an even $0\leq i\leq 2d=2k(n-k)$.
The associated Hodge polygon $\HP_{i}(G)$ of weight $i$ is the segment of slope $i/2$ connecting the points $(0,0)$ and $(b_{i}, ib_{i}/2)$.
On the other hand, according to a classification result on endomorphisms of the cohomology ring of $G$ due to Hoffman \cite{Hoffman84}, the map $f^*|_{H^{i}_{\et}(G, \bQ_\ell)}$ is either multiplication by $q^{i/2}$, or the composition of multiplication by $q^{i/2}$ with an involution (the latter case occurring only when $n=2k$).\footnote{Hoffman's key idea is to apply the hard Lefschetz theorem to solve a homogeneous system of linear equations, a method that carries over verbatim to positive characteristic. In a recent joint work with Jiang \cite{HJ25}, unaware of \cite{Hoffman84} at the time, we also employed this idea to establish the vanishing of certain intersection numbers.}
In either case, the eigenvalues of $f^*|_{H^{i}_{\et}(G, \bQ_\ell)}$ are precisely $\pm q^{i/2}$.
It follows that the associated Newton polygon $\NP_{\nu_\ell}(P_{i})$ or $\NP_{\nu_p}(P_{i})$ of the characteristic polynomial $P_i(t)$ consists of a single segment of slope $i/2$ and length $b_{i}$, and thus coincides with the Hodge polygon $\HP_{i}(G)$.
\end{proof}

To conclude, for the reader's convenience, we include an example of an abelian surface to illustrate \cref{prop:AV&Gr}.

\begin{example}[{cf.~\cite[Example~7.1]{MZ18a}}]
\label{example:MZ}
Let $X$ be the self-product $E^2$ of an elliptic curve $E$ over $\bk$, and let $f$ be an endomorphism of $X$ defined by the matrix
\[
\begin{pmatrix} 
1 & -5 \\
1 & 1 
\end{pmatrix}.
\]
This defines a $6$-polarized endomorphism of $X$.
For abelian varieties, it suffices to consider the action of $f$ on $H_{\et}^1(X, \bQ_\ell)$.
The characteristic polynomial of $f^*|_{H_{\et}^1(X, \bQ_\ell)}$ is
\[
P_1(t) = (t^2 - 2t + 6)^2 = t^4 - 4t^3 + 16t^2 - 24t + 36.
\]
The associated Newton polygon $\NP_{\nu_2}(P_1)$ is a single segment connecting the points $(0,0)$ and $(4,2)$.
The Newton polygon $\NP_{\nu_3}(P_1)$ coincides with the Hodge polygon $\HP_1(X)$ of weight $1$, which consists of two segments connecting the points $(0,0)$, $(2,0)$, and $(4,2)$.
In either case, we see that $\NP_{\nu_2}(P_1) \geq \HP_1(X)$ and $\NP_{\nu_3}(P_1) \geq \HP_1(X)$.
\end{example}

\section*{Acknowledgments}

The author was supported by NSFC Grant No.~12371045 and a start-up grant from Nanjing University.
This work was carried out during the author's visits to CIRM in Marseille in January 2025 and to the University of Waterloo in Winter 2025.
He was grateful to Jason Bell, Charles Favre, Matthew Satriano, Junyi Xie, and Yishu Zeng for stimulating discussions.
The author is grateful to the referees for their helpful comments and suggestions to improve the paper.

\providecommand{\bysame}{\leavevmode\hbox to3em{\hrulefill}\thinspace}
\providecommand{\MR}{\relax\ifhmode\unskip\space\fi MR }
\providecommand{\MRhref}[2]{%
  \href{http://www.ams.org/mathscinet-getitem?mr=#1}{#2}
}
\providecommand{\href}[2]{#2}


\begin{thebibliography}{MOA11}

\bibitem[BO78]{BO78}
Pierre Berthelot and Arthur Ogus, \emph{Notes on crystalline cohomology},
  Princeton University Press, Princeton, N.J.; University of Tokyo Press,
  Tokyo, 1978. \MR{0491705}

\bibitem[Del74]{Deligne74}
Pierre Deligne, \emph{La conjecture de {W}eil. {I}}, Inst. Hautes \'{E}tudes
  Sci. Publ. Math. \textbf{43} (1974), 273--307. \MR{0340258}

\bibitem[Del80]{Deligne80}
\bysame, \emph{La conjecture de {W}eil. {II}}, Inst. Hautes \'{E}tudes Sci.
  Publ. Math. \textbf{52} (1980), 137--252. \MR{601520}

\bibitem[HJ25]{HJ25}
Fei Hu and Chen Jiang, \emph{An upper bound for polynomial volume growth of
  automorphisms of zero entropy}, to appear in Peking Math. J. (2025+), 27 pp.,
  \doi{10.1007/s42543-025-00106-1}.

\bibitem[Hof84]{Hoffman84}
Michael Hoffman, \emph{Endomorphisms of the cohomology of complex
  {G}rassmannians}, Trans. Amer. Math. Soc. \textbf{281} (1984), no.~2,
  745--760. \MR{722772}

\bibitem[Hu24]{Hu24}
Fei Hu, \emph{Eigenvalues and dynamical degrees of self-maps on abelian
  varieties}, J. Algebraic Geom. \textbf{33} (2024), no.~2, 265--293.
  \MR{4705374}

\bibitem[Kle68]{Kleiman68}
Steven~L. Kleiman, \emph{Algebraic cycles and the {W}eil conjectures}, Dix
  expos\'es sur la cohomologie des sch\'emas, Adv. Stud. Pure Math., vol.~3,
  North-Holland, Amsterdam, 1968, pp.~359--386. \MR{292838}

\bibitem[KM74]{KM74}
Nicholas~M. Katz and William Messing, \emph{Some consequences of the {R}iemann
  hypothesis for varieties over finite fields}, Invent. Math. \textbf{23}
  (1974), 73--77. \MR{0332791}

\bibitem[Maz72]{Mazur72}
B.~Mazur, \emph{Frobenius and the {H}odge filtration}, Bull. Amer. Math. Soc.
  \textbf{78} (1972), 653--667. \MR{330169}

\bibitem[Maz73]{Mazur73}
\bysame, \emph{Frobenius and the {H}odge filtration (estimates)}, Ann. of Math.
  (2) \textbf{98} (1973), 58--95. \MR{321932}

\bibitem[MOA11]{MOA11}
Albert~W. Marshall, Ingram Olkin, and Barry~C. Arnold, \emph{Inequalities:
  theory of majorization and its applications}, second ed., Springer Series in
  Statistics, Springer, New York, 2011. \MR{2759813}

\bibitem[MZ18]{MZ18a}
Sheng Meng and De-Qi Zhang, \emph{Building blocks of polarized endomorphisms of
  normal projective varieties}, Adv. Math. \textbf{325} (2018), 243--273.
  \MR{3742591}

\bibitem[Ser60]{Serre60}
Jean-Pierre Serre, \emph{Analogues k\"{a}hl\'{e}riens de certaines conjectures
  de {W}eil}, Ann. of Math. (2) \textbf{71} (1960), 392--394. \MR{0112163}

\bibitem[Suh12]{Suh12}
Junecue Suh, \emph{Symmetry and parity in {F}robenius action on cohomology},
  Compos. Math. \textbf{148} (2012), no.~1, 295--303. \MR{2881317}

\bibitem[SZ16]{SZ16}
Shenghao Sun and Weizhe Zheng, \emph{Parity and symmetry in intersection and
  ordinary cohomology}, Algebra Number Theory \textbf{10} (2016), no.~2,
  235--307. \MR{3477743}

\bibitem[Tat64]{Tate-WoodsHole}
John~T. Tate, \emph{Algebraic cohomology classes}, 1964, Notes for the 1964
  Summer Institute in Algebraic Geometry held in Woods Hole, MA,
  \url{https://www.jmilne.org/math/Documents/}.

\bibitem[Tat65]{Tate65}
\bysame, \emph{Algebraic cycles and poles of zeta functions}, Arithmetical
  {A}lgebraic {G}eometry ({P}roc. {C}onf. {P}urdue {U}niv., 1963), Harper \&
  Row, New York, 1965, pp.~93--110. \MR{0225778}

\bibitem[Xie24]{Xie-GWRH}
Junyi Xie, \emph{Numerical spectrums control cohomological spectrums}, preprint
  (2024), 18 pp., \arxiv{2412.01216v2}.

\end{thebibliography}
\end{document}